\documentclass[11pt]{amsart}
\usepackage{amsmath,amsthm,amsfonts,amssymb,amscd}
\usepackage[dvips]{graphicx}
\usepackage{psfrag}
\usepackage{xcolor}

\theoremstyle{plain}
\newtheorem*{thm*}{Theorem}
\newtheorem{thm}{Theorem}[section]
\newtheorem{prop}[thm]{Proposition}
\newtheorem{lemma}[thm]{Lemma}

\newtheorem{maintheorem}{Theorem}

\theoremstyle{definition}
\newtheorem{remark}[thm]{Remark}

\newtheorem{claim}{Claim}

\renewcommand{\epsilon}{\varepsilon}

\newcommand{\bR}{\mathbb{R}}
\newcommand{\bN}{\mathbb{N}}

\DeclareMathOperator{\crit}{\rm Crit}

\newcommand{\R}{\mathbb{R}}

\newcommand{\Mf}{\mathfrak{M}(f)}

\newcommand{\comp}{\operatorname{Comp}}

\begin{document}

\title[Entropy of irregular points]{Entropy of irregular points\\ for some dynamical systems}

\author[K.~Gelfert]{Katrin Gelfert}
\address{Instituto de Matem\'atica Universidade Federal do Rio de Janeiro, Av. Athos da Silveira Ramos 149, Cidade Universit\'aria - Ilha do Fund\~ao, Rio de Janeiro 21945-909,  Brazil}\email{gelfert@im.ufrj.br}

\author[M. J.~Pacifico]{Maria Jos\'e Pacifico}
\address{Instituto de Matem\'atica Universidade Federal do Rio de Janeiro, Av. Athos da Silveira Ramos 149, Cidade Universit\'aria - Ilha do Fund\~ao, Rio de Janeiro 21945-909,  Brazil}\email{pacifico@im.ufrj.br}

\author[D.~Sanhueza]{Diego Sanhueza}
\address{Instituto de Matem\'atica Universidade Federal do Rio de Janeiro, Av. Athos da Silveira Ramos 149, Cidade Universit\'aria - Ilha do Fund\~ao, Rio de Janeiro 21945-909,  Brazil}\email{sanhueza.diego.a@gmail.com}

\begin{thanks}
{This research has been supported [in part] by CAPES -- Finance Code 001 and by CAPES- and CNPq-grants. MJP was partially supported by FAPERJ}
\end{thanks}

\maketitle
\DeclareGraphicsExtensions{.jpg,.pdf,.mps,.png}

\begin{abstract} We derive sufficient conditions for a dynamical systems to have a set of irregular points with full topological entropy. Such conditions are verified for some nonuniformly hyperbolic systems such as positive entropy surface diffeomorphisms and rational functions on the Riemann sphere.
\end{abstract}

\section{Introduction}\label{s-introd}

Given a continuous map $f\colon X\to X$ on a compact metric space $X$ and a continuous observable $\varphi\colon X\to\R$, its Birkhoff average along the orbit 
$$
	\lim_{n\to\infty}\frac{1}{n}\sum_{j=0}^{n-1}\varphi(f^j(x))
$$
(for $x\in X$ for which this limit exists) play an important role because of their intimate relation with convergence in the weak$\ast$ topology. Recall that the set of $\varphi$-\emph{irregular points} for which the Birkhoff averages do not converge, 
\[
	\hat{X}(f,\varphi)
	:= \Big\{x\in X\colon \lim_{n\to\infty}\frac{1}{n}\sum_{j=0}^{n-1}\varphi(f^j(x))
	\text{ does not exist}\Big\},
\]	
 is universally null. Ruelle \cite{Ruelle(2001)} coined the term of ``points with historical behaviour'' because they  trace the history of the system, whereas points whose Birkhoff's sum converge only warn average behaviour. Although $\hat{X}(f,\varphi)$ is not detected by any invariant measure, it can be ``large'' from another point of view such as, for example, fractal dimension, entropy, or general topology. This type of question is typical in multifractal analysis.  Irregular points form an essential part of the \emph{multifractal decomposition} of $X$ (relative to $\varphi$),
\[
	X
	= \bigcup_{\alpha\in\mathbb R\cup\{\infty\}} X_\alpha(f,\varphi)\,\dot\cup \,\hat{X}(f,\varphi),
\]  
where $X_\alpha(f,\varphi)$ denotes the set of \emph{$\varphi$-regular points} with average $\alpha$,
\[
	X_\alpha(f,\varphi)
	:=\Big\{x\in X\colon \displaystyle\lim_{n\to\infty} \frac1{n}\displaystyle\sum_{j=0}^{n-1} 
\varphi(f^j(x))=\alpha\Big\}.
\] 
Denote also the \emph{set of irregular points} of $f$ by
\[
	\hat X(f)
	:= \bigcup\Big\{ \hat X(f,\varphi)\colon \varphi\colon X\to\bR\text{ continuous}\Big\}.
\]

Our focus here will be on entropy. Although any of the above defined sets is $f$-invariant, in general it is noncompact and we  rely on the concept of topological entropy introduced in \cite{Bowen(1973)}. We denote by  $h(f,A)$ the \emph{topological entropy} of $f$ on $A\subset X$. 
The following estimates hold true in general
\[
	 0 \le h(f,\hat X(f,\varphi)) \le h(f,\hat{X}(f)) \le h(f,X).
\]
Each inequality can be strict. Recall, for instance, the example of a minimal dynamical system (hence satisfying $\hat X(f)=\emptyset$) with positive entropy in \cite{Hahn-Katznelson(1967)}. To our best knowledge, there is no nontrivial (that is, say, topologically transitive) example for which $0<h(f,\hat{X}(f)) < h(f,X)$.

Previous approaches to analyze of the regular and the irregular part of the spectrum commonly require ``orbit-gluing properties''. For example, \cite{Pesin-Pitskel(1984),Fan-Feng(2000)} consider the case of a subshift of finite type. In \cite{Takens-Verbitskiy(2003)}, just assuming the specification property of $f$, there is stated a \emph{restricted variational principle} 
\[
	h(f,X_\alpha(f,\varphi))
	=\sup\Big\{h_\mu(f)\colon 
		\mu\ f\mbox{-invariant,}\int{\varphi}\,d\mu=\alpha\Big\}.
\]
On the other hand, ``chaotic dynamics'' quite commonly gives rise to irregular sets which are dense and have full entropy. For a full shift of two symbols $\sigma\colon\Sigma_2\to\Sigma_2$, by \cite[Lemma 6]{Pesin-Pitskel(1984)} it holds $h(\sigma,\hat{\Sigma}_2(\sigma))=h(\sigma,\Sigma_2)$. If $f$ satisfies the specification property, it is a consequence of \cite[Theorem 4]{Sigmund(1974)} (see also \cite[Theorem 3]{Dateyama(1981)}) that $\hat{X}(f)$ is nonempty and of \cite{Ercai-Kupper-Lin(2005)} that it has full entropy.

The specification property roughly says that given any number of arbitrarily long orbit segments, there exists an orbit which stays $\varepsilon$-close to each segment and between segments there are only a bounded number of iterations whose number only depends on $\varepsilon$ (we refer to \cite{Sigmund(1974),Thompson(2012)} for the full definition). For example, for a basic set of an Axiom A diffeomorphism, the existence of a Markov partition enables a symbolic description of  orbits and shadowing permits to ``connect'' orbit segments which were symbolically coded previously. As such partitions can be chosen with arbitrarily small diameter, this guarantees the existence of ``arbitrarily specified orbits''.  

In general, specification, or any of it's weaker versions, does not hold. We will illustrate that it is also often not required in this strong sense to deduce ``maximal historic behavior''.  Here the key observation is that such strong hypotheses are not required globally, but only for  invariant subsystems whose entropies are sufficiently large and that entropy in some sense is ``a local quantifier''.

We start by considering the set of irregular points and not \emph{a priori} fixing any observable, assuming  hyperbolicity.
 
\begin{maintheorem}\label{t-AxiomA} 
For any Axiom A  $C^{1}$-diffeomorphism $f\colon X\to X$ on a $n$-dimensional closed manifold $X$, $n\geq2$, it holds
$$
	h(f,\hat{X}(f))
	= h(f,X).
$$
\end{maintheorem}

In a ``nonuniformly hyperbolic context'', the proof of the next result takes into consideration approximations in the weak$\ast$ topology and in entropy of positive-entropy ergodic measures by horseshoes.

\begin{maintheorem}\label{TeoKatokThompson} 
For any $C^{1+\alpha}$-diffeomorphism $f\colon X\to X$ on a closed surface $X$ it holds
$$
	h(f,\hat{X}(f))
	=h(f,X).
$$
\end{maintheorem}

A version of Theorem \ref{TeoKatokThompson} was obtained independently in \cite[Corollary 2.4]{Barrientos-Nakano-Raibekas-Roldan(2021)}.

To guarantee that $h(f,\hat X(f,\varphi))$ is positive (or even large), some hypotheses on the observable are absolutely necessary. Indeed, if $\varphi$ is cohomologous to a constant, then $\hat{X}(f,\varphi)$ is empty (and hence has zero entropy). If $\hat{X}(f,\varphi)$ is nonempty, for a topologically mixing subshift of finite type (or any topologically conjugate system), by \cite{Barreira-Schmeling(2000)} this set  carries full topological entropy. This result was generalized to maps satisfying the almost specification property (\cite{Thompson(2012)}, see also  \cite{Pacifico-Sanhueza(2019)} for the context of flows) or the orbit gluing property (see \cite{Lima-Varandas(2020)} and further references therein). 
Moreover, by \cite{Dong-Tian-Yuan(2015)}, assuming the asymptotic average shadowing property (AASP for short) and a certain condition on the measure center, the set $\hat{X}(f,\varphi)$  is either residual or empty. It is not known if AASP in general implies that $\hat{X}(f,\varphi)$, if nonempty, has full entropy. 

The following result provides a sufficient condition on $\varphi$ to ``detect'' historic behavior and forces $h(f,\hat{X}(f,\varphi))$ to be large (or even full). Let us denote by $\Mf$ the set of $f$-invariant Borel probability measures.

\begin{maintheorem}\label{Metateorema} 
Let $f\colon X\to X$ and $\varphi\colon X\to\mathbb R$ be continuous functions on a compact metric space $X$ such that there is a sequence $(\Gamma_n)_n$ of compact $f$-invariant subsets of $X$ and numbers $\ell_n\in\bN$ such that $f^{\ell_n}|_{\Gamma_n}$ has the almost specification property and 
\begin{equation}\label{condTC}
	\inf_{\mu\in\mathfrak M(f|_{\Gamma_n})}\int\varphi\,d\mu
	<\sup_{\mu\in\mathfrak M(f|_{\Gamma_n})}\int\varphi\,d\mu.
\end{equation}
 Then, it holds
$$
	\limsup_{n\to\infty}h(f,\Gamma_n)
	\le h(f,\hat{X}(f,\varphi))
	\le h(f,\hat X(f)) =
	h(f,X) .
$$
\end{maintheorem}

\begin{remark} In many cases (such as, for example, under the hypotheses of Theorem \ref{JuliaIrre} below), the sets $\Gamma_n$ to verify the hypotheses of Theorem \ref{Metateorema} can be chosen to be nested, that is, $\Gamma_n\subseteq \Gamma_{n+1}$, and so it suffices to show that \eqref{condTC} holds for some $n$ (and hence for all $n'\ge n$).
\end{remark}

\emph{A priori}, fixing some observable $\varphi\colon X\rightarrow\R$, the value $h(f,\hat{X}(f,\varphi))$  may be much smaller than $h(f,X)$. To exemplify this, recall that  a $C^1$ diffeomorphism $f$ is Axiom A if its  nonwandering set $\Omega(f)$ is hyperbolic and the periodic points of $f$ are dense in $\Omega(f)$. Now, given a $C^1$ Axiom A diffeomorphism $f$ 
with nontrivial spectral decomposition $\Omega(f)=\Omega_1\cup\cdots \cup\Omega_\ell$ such that $h(f,\Omega_i)<h(f,\Omega_j)$ for some index pair $i\ne j$, then for  $\varphi\colon M\to\mathbb{R}$ such that $\varphi|_{\Omega_i}$ is not cohomologous to a constant and $\varphi|_{\Omega_j}\equiv0$ it holds 
\[ \hat{X}(f,\varphi)\neq\emptyset \quad \mbox{and} \quad 
	h(f,\hat{X}(f,\varphi))
	<\max_{j\ne i}h(f,\Omega_j)
	= h(f,X).
\]

Finally, we invoke Theorem \ref{Metateorema} in a setting which can also be considered ``nonuniformly hyperbolic''. Note that this setting, in particular in the case where $f$ is a rational function has a critical point inside its Julia set, is in general very far from being hyperbolic or satisfying specification. A classical example is the Ulam-von Neumann map $z\mapsto z^2-2$ with Julia set being the closed interval $[-2,2]$ and critical point $z=0$.

\begin{maintheorem}\label{JuliaIrre} 
Suppose that $f$ is a rational function of degree $d\geq2$ on the Riemann sphere and $J(f)$ its Julia set. Then
\[
	h(f,\hat J(f)) 
	= h(f,J(f))
	= \log d.
\]
Moreover,  there are no critical  points in $J(f)$, then for $\varphi:=\log|f'|$ it holds 
\[
	\hat J(f,\varphi)\ne\emptyset
	\quad\text{ and }\quad
	h(f,\hat J(f,\varphi)) 
	=	h(f,\hat J(f)) 
	=h(f,J(f))
\]
if, and only if,  $f$ is not of the form $f(z)=z^{\pm d}$.
\end{maintheorem}

In Theorem \ref{JuliaIrre}, for the second claim we assume that there are no $f$-critical points in $J(f)$ in order to guarantee that $\varphi$ is continuous. An analogous statement in the general case, including critical points and extending the concept of $\varphi$-irregular points to general observables, can be derived following, for example,  techniques in \cite{GelPrzRamRiv:13}.

In Section \ref{s-topentropy} we recall some preliminaries. In Section \ref{TEDiferenciavel} we prove Theorems  \ref{t-AxiomA},  \ref{TeoKatokThompson}, and \ref{Metateorema}.  In Section \ref{Sec-Julia} we study rational maps and prove Theorem \ref{JuliaIrre}.

\section{Preliminaries}\label{s-topentropy}

Note that $\hat{X}(f,\varphi)$ is empty whenever $\varphi$ is ``essentially constant''. 
More precisely, consider the space $C(X)$ of continuous observables $\varphi\colon X\to\mathbb{R}$, equipped with the usual sup-norm $\lVert\varphi\rVert:=\sup\lvert\varphi\rvert$.
Two functions $\varphi,\psi\in C(X)$ are \emph{cohomologous} (with respect to $f$) if there exists $u\in C(X)$ such that $\psi=\varphi+u-u\circ f$. Any function which is cohomologous to a constant function is a \emph{coboundary}. The following facts are immediate. 

\begin{prop} 
For every $\varphi\in C(X)$, $c_1,c_2\in\mathbb{R}$, $c_1\neq0$, and $k\in\mathbb{N}$, 
$\hat{X}(f,\varphi)=\hat{X}(f,c_1\varphi+c_2)=\hat{X}(f,\varphi\circ{f^k})$. If $\varphi,\psi\in C(X)$ are cohomologous, then $\hat{X}(f,\varphi)=\hat{X}(f,\psi)$. 
For every coboundary $\varphi$, it holds $\hat X(f,\varphi)=\emptyset$.
\end{prop}

\begin{remark}\label{rem:cohomo}
For $\varphi\in C(X)$ the following facts are equivalent:
\begin{itemize}
\item[1)] $\inf_{\mu\in\Mf}\int\varphi\,d\mu<\sup_{\mu\in\Mf}\int\varphi\,d\mu$,
\item[2)] $\varphi$ is not in the closure of the subset of coboundaries,
\item[3)] $\frac1n\sum_{j=0}^{n-1}\varphi\circ f^j$ does not converge pointwise to a constant.
\end{itemize}
See, for example \cite[Lemma 2.1]{Thompson(2012)}, where further properties are stated. It is clear that if $\hat X(f,\varphi)\ne\emptyset$, then property 3) (and hence any other) holds true. Assuming the almost specification property, by \cite[Theorem 4.1]{Thompson(2012)} property 1) (and hence any other) implies $\hat X(f,\varphi)\ne\emptyset$.
\end{remark}

In general, $f$ can have a complicated dynamics and it can be more convenient to analyze certain subsystems. To this end, we state the following lemma which is straightforward to show. 

\begin{lemma}\label{lem:Irregularesfk} 
For each $\ell\in\mathbb N$ and $\varphi\in C(X)$, it holds
$$
\hat{X}(f,\varphi)=\bigcup_{j=0}^{\ell-1}\hat{X}(f^\ell,\varphi\circ f^{j}).
$$
In particular, $\hat{X}({f^\ell},\varphi)\subseteq \hat{X}(f,\varphi)$ and $\hat{X}(f)=\hat{X}(f^\ell)$.
\end{lemma}

Observe that $\hat{X}({f^\ell},\varphi)$ can be a proper subset of $\hat{X}(f,\varphi)$.

\section{Proof of Theorems \ref{t-AxiomA},  \ref{TeoKatokThompson}, and \ref{Metateorema}}\label{TEDiferenciavel}

\begin{proof}[Proof of Theorem \ref{Metateorema}]
For each $n\geq1$,  \cite[Theorem 4.1]{Thompson(2012)} implies
\[
	h(f^{\ell_n},\Gamma_n)
	= h\big(f^{\ell_n}|_{\Gamma_n},\hat{X}(f^{\ell_n}|_{\Gamma_n},\varphi|_{\Gamma_n})\big)
\]
and hence Lemma \ref{lem:Irregularesfk} together with $\hat{X}(f|_{\Gamma_n},\varphi|_{\Gamma_n})\subseteq \hat{X}(f,\varphi)$ and monotonicity of entropy gives
\[
	h(f,\Gamma_n)
	= h\big(f|_{\Gamma_n},\hat{X}(f^{\ell_n}|_{\Gamma_n},\varphi|_{\Gamma_n})\big)
	\le h\big(f,\hat{X}(f,\varphi)\big)
	\le h(f,\hat{X}(f))
	\le h(f,X).
\]
Letting $n\to\infty$, this implies the claim.
\end{proof}

\begin{proof}[Proof of Theorem \ref{TeoKatokThompson}] 
If $h(f,X)=0$, the result is immediate.

If $h(f,X)>0$, by the variational principle for entropy, there is a sequence of ergodic measures $(\mu_n)_{n\geq1}$ such that $h_{\mu_n}(f)\to h(f,X)$. By Ruelle's inequality, any $\mu_n$ with positive entropy is hyperbolic. Let us hence consider $\mu$ hyperbolic ergodic with entropy arbitrarily close to $h(f,X)$. By Katok's horseshoe construction \cite[Theorem S.5.9]{Katok-Hasselblatt}, for every $\varepsilon>0$, there exists a basic set $\Gamma\subset M$ and $m\in\bN$ such that $f^m|_\Gamma$ is topologically mixing and hence has the specification property, and $h_{\mu}(f)-\varepsilon<h(f,\Gamma)$.
Hence,
\[\begin{split}
	h(f,\Gamma) 
	& =  \frac{1}{m}h(f^m,\Gamma)
	 =  \frac{1}{m} h(f^m,\hat{X}(f^m|_\Gamma))\\
	\text{\tiny{by Lemma \ref{lem:Irregularesfk}}}\quad
	& \leq \frac{1}{m} h(f^m,\hat{X}(f|_\Gamma))\\
	& =  h(f|_\Gamma,\hat{X}(f|_\Gamma))
	\le h(f,\hat{X}(f)).
\end{split}\]
Choosing $\mu$ with entropy arbitrarily close to $h(f,X)$ implies the claim.
\end{proof}

In higher dimensions there may not exist any hyperbolic ergodic measure and hence the strategy of the proof of  Theorem \ref{TeoKatokThompson} does not work. In an \emph{a priori} hyperbolic context, we invoke \cite[Theorem 4.1]{Thompson(2012)} (or \cite{Barreira-Schmeling(2000)}) directly.

\begin{proof}[Proof of Theorem \ref{t-AxiomA}]
By hypothesis, we can invoke the spectral decomposition theorem. Hence, the set of nonwandering points splits as
$$
\Omega(f)=\Omega_1\cup\ldots \cup\Omega_m,
$$
where for every $i=1,\ldots,m$ there is $\ell_i\in\bN$ such that $\Omega_i=X_{1,i}\cup\ldots \cup X_{\ell_i,i}$ such that  $f^{\ell_i}|_{X_{k,i}}$ is topologically mixing, for each $1\leq k\leq \ell_i$, and hence has the specification property. As
\begin{equation*}
	h(f,X)
	= \max_i h(f,\Omega_i)
	= \max_i\frac{1}{\ell_i}h(f^{\ell_i},X_{k,i})
	\leq h(f,\hat{X}(f)),
\end{equation*}
this proves the claim.
\end{proof}

\section{Irregular points in Julia sets and proof of Theorem \ref{JuliaIrre}}\label{Sec-Julia}

In this section, let $f\colon \overline{\mathbb{C}}\to\overline{\mathbb{C}}$ be a rational function of degree $d:=\deg(f)\geq2$ and consider its Julia set $J=J(f)$. 
 We estimate the entropy of the set of $\varphi$-irregular points for the \emph{geometric potential} $\varphi(z):=\log\,\lvert f'\rvert$. This observable, along with its scaled version $t\varphi$ for $t\in\bR$, plays an important role in the thermodynamics formalism (see \cite{PrzUrb:10,Urbanski(2003)} and references therein). Its Birkhoff averages are simply the Lyapunov exponents. 
If $f$ has no critical points in $J$ then $\varphi$ is continuous on $J$.

Recall that $J$ is nonempty, compact, and coincides with the closure of the set of repelling periodic points. Moreover, $f|_{J}$ is topologically exact and its entropy equals $h(f,J)=\log d$   (see, for example,  \cite{Beardon,Ljubich(1983)}). 

A compact $f$-invariant set $R\subset J$ is a \emph{uniformly expanding repeller} if  $f|_R$ is \emph{hyperbolic}, that is, there is $n\in\mathbb N$ so that
\[
	\inf\{|(f^n)'(z)|\colon z\in R\}>1,
\]
topologically transitive, and isolated, that is, there exists an open neighborhood $U\supset R$ such that $f^n(z)\in U$ for all $n\ge0$ implies $z\in R$.
Note that if a uniformly expanding repeller $R$ is such that $f^\ell|_R$ is topologically mixing then $f^\ell|_R$ has the specification property.

A rational function $f$ is \emph{hyperbolic} if $f|_J$ is hyperbolic. Every rational function without critical points in $J$ is either hyperbolic or admits a \emph{parabolic point}, that is, a periodic point $z=f^p(z)\in J$ such that $(f^p(z))'(z)$ is a root of  unity (see \cite[Theorem 3.2]{Urbanski(2003)}).
Given an ergodic Borel probability measure $\mu\in\mathfrak{M}_{\mathrm{erg}}(f)$ consider its Lyapunov exponent
\[
	\mathcal{X}(\mu)
	:= \int \log\,\lvert f'\rvert\,d\mu
	= \int\varphi\,d\mu.
\]
Let $\alpha^-:=\inf\{\mathcal{X}(\nu)\colon \nu\in\mathfrak{M}_{\mathrm{erg}}(f)\}$ and $\alpha^+:=\sup\{\mathcal{X}(\nu)\colon \nu\in\mathfrak{M}_{\mathrm{erg}}(f)\}$. Let $\mu_0$ be the (unique, hence ergodic) measure of maximal entropy. Recall
\[
	0
	< \log d
	\le \mathcal{X}(\mu_0)
	\leq \alpha^+.
\]

The following preliminary result can be proved in several different ways, for example building upon nowadays available description of the spectrum of Lyapunov exponents and approximation techniques (see, for example, \cite[Chapter 11]{PrzUrb:10} and references therein). We provide a direct proof.

\begin{lemma}\label{Katrin-Completitude} 
If $f|_J$ has a parabolic point, then there is a hyperbolic ergodic measure $\nu$ supported on a periodic orbit satisfying $0<\mathcal{X}(\nu)<\mathcal X(\mu_0)$. 
\end{lemma}

\begin{proof}
First note that, as $f|_{J}$ is topologically exact, for every $\delta>0$ there is $N(\delta)\in\bN$ such that for every $x\in J$ and $N\geq N(\delta)$ it holds $f^N(B(x,\delta))\supset J$.
 
By contradiction, suppose that all hyperbolic periodic points have the same exponent  $\alpha:=\mathcal X(\mu_0)$.

Given $\varepsilon\in(0,\alpha/4)$, topological exactness together with our assumptions implies that there is $\delta_1>0$ such that for every hyperbolic periodic point $w=f^p(w)$ the set $U=B(w,\delta_1)$ satisfies $f^p(U)\supset U$ and
\begin{equation}\label{refd1}
\frac{|(f^p)'(x)|}{|(f^p)'(y)|}\leq e^{p\varepsilon}
\end{equation}
for all $x,y\in U$. Let $N_1=N(\delta_1)$.

Let $\crit(f)$ be the set of all the critical points of  $f\colon \overline{\mathbb{C}}\to\overline{\mathbb{C}}$. Fix a parabolic point $z=f^q(z)$ and let $\delta_0>0$ so that $\cup_{k=0}^{q-1}f^k(\crit(f))$  and $B(z,\delta_0)$ are disjoint. Hence, for any $\delta\in(0,\delta_0)$ and any preimage of $V=B(z,\delta)$ by $f^q$ is univalent; denote by $\comp_zf^{-q}(V)$ its connected component containing $z$. By the Koebe Distortion Theorem,  there is $C(\delta)>1$ satisfying $C(\delta)\to1$ as $\delta\to0$ so that for all $x,y\in V':=\comp_zf^{-q}(V)$ it holds
\begin{equation}\label{refd2}
\frac{|(f^q)'(x)|}{|(f^q)'(y)|}\leq C(\delta).
\end{equation}
Fix some $\delta\in (0,\delta_0)$ and let $N_2=N(\delta)$.

Recall that there are repelling periodic points of arbitrarily large period. 
We choose a periodic point $w=f^p(w)$ with period $p$ large enough, whose choice is specified below. Let
$$
W':=\comp_zf^{-(N_2+q)}(U)\subseteq V'.
$$
By the above, we have
$$
	W'
	\subset f^{N_1}(U)
	\subset f^{N_1+p}(U)
$$
and thus there exists a periodic point  $x=f^k(x)\in W'$ with period $k=q+N_1+p+N_2$. By the distortion estimates \eqref{refd1} and \eqref{refd2}, it holds
$$
	|(f^k)'(x)|
	\leq 1\cdot C(\delta)\cdot  |(f^p)'(w)|\cdot e^{p\varepsilon}
		\cdot D^{N_1+N_2},
	 \quad\text{ where }\quad
	 D:=\max_J|f'|,
$$
together with the analogous lower bound. In particular, $x$ is hyperbolic and
\[\begin{split}
	\frac{1}{k}\log|(f^k)'(x)|
	&\leq \frac{1}{q+p+N_1+N_2}(\log{C(\delta)}+(N_1+N_2)\log{D})+\\
	&\phantom{\le}
	 +\frac{p}{q+p+N_1+N_2}\alpha+\frac{p}{q+p+N_1+N_2}\varepsilon.
\end{split}\]
Now choosing $w$ with period $p$ large enough, we obtain that
$\log|(f^k)'(x)|/k<\alpha/2$, contradiction.
\end{proof}

\begin{proof}[Proof of Theorem \ref{JuliaIrre}] 
We start by proving the first claim about the entropy of the set of irregular points. 
By Ruelle's inequality, the ergodic measure of maximal entropy, $\mu_0$, is hyperbolic and satisfies $\int\log\,\lvert f\rvert\,d\mu_0\ge\log d>0$.
Hence, by \cite[Theorem 11.6.1]{PrzUrb:10}, there exist a sequence $(\Gamma_n)_n$ of uniformly expanding repellers satisfying $h(f,\Gamma_n)\to h(f,J)$ as $n\to\infty$ and numbers $m_n\in\bN$ such that $f^{m_n}|_{\Gamma_n}$ mixing. In particular, $f^{m_n}|_{\Gamma_n}$ has the specification property.  
As $\Gamma_n$ has positive entropy, $\mathfrak{M}(f,\Gamma_n)$ is not a singleton and, in particular, there exists $\phi_n\in C(J)$ distinguishing measures in $\mathfrak{M}(f,\Gamma_n)$ and hence satisfying \eqref{condTC}. Invoking \cite[Lemma 2]{Gelfert-Przytycki-Rams(2010)}, in fact $(\Gamma_n)_n$ can be chosen to be increasing, that is, $\Gamma_n\subset\Gamma_{n+1}\subset\ldots$. Thus, there exists some common $\phi\in C(J)$ satisfying \eqref{condTC} for every $n\in\bN$. Thus, by Theorem \ref{Metateorema}, we obtain
\[
	h(f,\hat J(f,\phi))=h(f,\hat J(f))=h(f,J)=\log d.
\]

In the second part of this proof, we assume that $f$ has no critical points in $J$ and hence $\varphi=\log\,\lvert f'\rvert$ is continuous on $J$. 

If $f|_J$ is hyperbolic, then it has the shadowing property and topological exactness implies that $f|_J$ is topologically mixing. Hence, $f|_{J}$ satisfies the specification property (see, for example, \cite[Lemma 9]{Kwietniak-Oprocha(2012)}) and thus $h(f,\hat J(f,\varphi))=h(f,J)$ follows from \cite[Theorem 4.1]{Thompson(2012)}.

Otherwise, if $f|_J$ has a parabolic point, by Lemma \ref{Katrin-Completitude}, there is a $f$-invariant measure $\nu$ supported on a periodic hyperbolic orbit $\Gamma_\nu$ satisfying
$$
	0
	< \int\log|f'|\,d\nu
	=\mathcal X(\nu)
	<\mathcal X(\mu_0)
	=\int\log|f'|\,d\mu_0.
$$
By \cite[Theorem 11.6.1]{PrzUrb:10}, there exist mixing uniformly expanding Cantor repellers $\Gamma_n\subseteq J$ satisfying
\[
	h(f,\Gamma_n)\ge h(f,J)-\frac1n
\]
so that any invariant measure $\mu_n$ supported in  $\Gamma_n$ satisfies
\[
	\int\log|f'|\,d\nu<\int\log|f'|\,d\mu_0-\frac1n
	\le \int\log|f'|\,d\mu_n.
\]
In particular, for $n$ large, $\Gamma_\nu$ and $\Gamma_{\mu_n}$ are disjoint.
By \cite[Lemma 2]{Gelfert-Przytycki-Rams(2010)}, there exist transitive uniformly expanding Cantor repellers $\Lambda_n\supseteq \Gamma_n\cup\Gamma_\nu$. In particular, there is $\ell_n\in\mathbb N$ such that $f^{\ell_n}|_{\Lambda_n}$ satisfies the specification property, $\nu,\mu_n$ are supported in $\Lambda_n$, and
\[
	h(f,\Lambda_n)\ge h(f,J)-\frac1n.
\]
Hence, by Theorem \ref{Metateorema} it follows $h(f,\hat J(f,\varphi))=h(f,J)$. 

It remains to analyze the case when the spectrum is ``trivial''. 

\begin{claim}\label{lem:notcoh}
If $f$ has no critical points in $J$, then the following facts are equivalent:
\begin{itemize}
\item[(1)] $\hat J(f,\varphi)=\emptyset$,
\item[(2)] $\alpha^-=\alpha^+$, 
\item[(3)] $f$ is (conjugate to) $z\mapsto z^{\pm d}$.
\end{itemize}
\end{claim}

\begin{proof}
Check that if $f|_J$ is hyperbolic then $\alpha^-=\alpha^+$ if and only if $\varphi$ is a coboundary and, in particular by Remark \ref{rem:cohomo}, (2) and (1) are equivalent.
Recall that if $f|_J$ is hyperbolic and $\varphi$ a coboundary then by  \cite[Corollary in Section 7 and Proposition 8]{Zdunik(1990)}, $f$ is conjugate to $z^{\pm d}$. 
Clearly, if $f$ is conjugated to $z\mapsto z^{\pm d}$, then $\varphi\equiv\log d$ and $\alpha^-=\alpha^+$, hence (3) implies (2).
This proves equivalence of (1), (2), and (3) for $f$ hyperbolic.

Moreover, observe that none of the three cases occurs if $f$ has any parabolic point. Indeed, if $\lambda$ is the $f$-ergodic measure supported on a parabolic periodic orbit then (2) does not hold as
\[
	\alpha^-
	= \mathcal{X}(\lambda) 
	= 0
	< 
	\log d
	= \mathcal{X}(\mu_0)
	\leq \alpha^+.
\]
Clearly, (3) does not hold either. Finally, the above second part of the proof of the theorem implies that in this case $\hat J(f,\varphi)\ne\emptyset$.
\end{proof}
This proves the theorem.
\end{proof}

\bibliographystyle{plain}

\end{document}